\documentclass[11pt, reqno]{amsart}

\title{Subadditive bijections on the positive reals}

\usepackage[T1]{fontenc}
\usepackage{amsmath}
\usepackage{amssymb}
\usepackage{amsthm}
\usepackage[left=3.5cm, right=3.5cm, paperheight=11.8in]{geometry}
\usepackage{hyperref}
\usepackage{fancyhdr}
\usepackage{enumitem}
\usepackage{comment}
\usepackage{nicefrac}
\usepackage{mathrsfs}
\usepackage{bm}
\usepackage{graphicx}
\usepackage[utf8]{inputenc}
\usepackage{cancel}
\usepackage{mathtools}

\AtBeginDocument{%
   \def\MR#1{}
}

\newtheorem{thm}{Theorem}[section]

\newtheorem{prop}[thm]{Proposition}
\newtheorem{question}[thm]{Question}

\theoremstyle{definition} 
\newtheorem{defi}[thm]{Definition}
\let\olddefi\defi
\renewcommand{\defi}{\olddefi\normalfont}
\newtheorem{example}[thm]{Example}
\let\oldexample\example
\renewcommand{\example}{\oldexample\normalfont}
\newtheorem{rmk}[thm]{Remark}
\let\oldrmk\rmk
\renewcommand{\rmk}{\oldrmk\normalfont}

\author[P.~Leonetti]{Paolo Leonetti}
\address{
Universit\'{a} degli Studi dell’Insubria\\ via Monte Generoso 71 \\ Varese 21100\\ Italy}
\email{leonetti.paolo@gmail.com}
\urladdr{\url{https://sites.google.com/site/leonettipaolo}}

\keywords{Subadditive functions; subadditive bijections; functional inequalities; additive functions; discontinuous bijections.}

\subjclass[2020]{Primary: 39B62;
Secondary: 39B22, 26A15, 15A03.}

\hypersetup{
    pdftitle={Subadditive bijections on the positive reals},
    pdfauthor={Paolo Leonetti},
    pdfmenubar=false,
    pdffitwindow=true,
    pdfstartview=FitH,
    colorlinks=true,
    linkcolor=blue,
    citecolor=green,
    urlcolor=cyan
}

\providecommand{\MR}[1]{}
\providecommand{\bysame}{\leavevmode\hbox to3em{\hrulefill}\thinspace}
\providecommand{\MR}{\relax\ifhmode\unskip\space\fi MR }

\begin{document}

\begin{abstract} 
\noindent We show that there exists a subadditive bijection $f:(0,\infty)\to(0,\infty)$ such that 
$\liminf_{x\to 0}f(x)=0$ and $\limsup_{x\to 0}f(x)=1$. 
This answers a question by 
Matkowski and {\'S}wi{\k a}tkowski in [Proc. Amer. Math. Soc. \textbf{119} (1993), 187--197].
\end{abstract}
\maketitle
\thispagestyle{empty}

\section{Introduction and Main result}

Matkowski and {\'S}wi{\k a}tkowski proved in \cite[Corollary 2]{MR1176072} 
that a subadditive bijection $f: (0,\infty)\to (0,\infty)$ with $\lim_{x\to 0}f(x)=0$ has to be an homeomorphism of $(0,\infty)$, cf. also \cite{MR1088646}. 
In fact, it is known that many local properties of a subadditive function depend on its behaviour in a neighbourhood of the origin, see e.g. \cite{MR89373, MR2467621, MR36796}. 
This type of results are useful in converse Minkowski-type arguments and characterizations of $L^p$-norms (cf. Section \ref{sec:conclusions}), 
where one first obtains an auxiliary subadditive map and then needs to promote it to a well-behaved function, see e.g. \cite{MR1009994, MR1113646, 
MR2449349, Matkowski2013, MP1995, MR1088646, MR1234992}.

Among the discontinuous examples in this setting, Matkowski and {\'S}wi{\k a}tkowski proved in \cite[Theorem 2]{MR1176072} that there exists a subadditive bijection $f: (0,\infty)\to (0,\infty)$ such that 
$$
\liminf_{x\to 0}f(x)>0
\qquad \text{ and }\qquad 
\limsup_{x\to 0}f(x)<\infty, 
$$
cf. also \cite[p. 194]{MR1176072}. 
They also remark that \cite[Example 1]{MR1088646} shows the existence of a subadditive bijection $f: (0,\infty)\to (0,\infty)$ such that 
$$
\liminf_{x\to 0}f(x)=0 \qquad \text{and }\qquad \limsup_{x\to 0}f(x)=\infty. 
$$
Accordingly, they pose the following open question:
\begin{question}\label{q:matk2}
Does there exist a subadditive bijection $f: (0,\infty)\to (0,\infty)$ such that 
\begin{equation}\label{eq:mirkclaim}
\liminf_{x\to 0}f(x)=0 \qquad \text{and }\qquad 0<\limsup_{x\to 0}f(x)<\infty\,\,?
\end{equation}
\end{question}
In \cite[Theorem 3]{MR1176072}, they observed that if $f: (0,\infty)\to (0,\infty)$ is a subadditive bijection and $f^{-1}$ is bounded in a neighborhood of $0$ then $\lim_{x\to 0}f^{-1}(x)=0$. In light of this result, they concluded that Question \ref{q:matk2} ``seems to be rather difficult to decide.'' 

Finally, we provide a positive answer to Question \ref{q:matk2}. 
\begin{thm}\label{thm:mainjanusz}
There exists a subadditive bijection $f:(0,\infty)\to(0,\infty)$ such that 
$$
\liminf_{x\to0}f(x)=0
\qquad\text{and}\qquad
\limsup_{x\to0}f(x)=1.
$$
\end{thm}

As it turns out from our proof, the constructed map $f$ satisfies 
$$
\forall x>0, \qquad x<f(x)<x+1,
$$
cf. Section \ref{sec:final}. In particular, since $f^{-1}(x)\le x$, it follows also that $\lim_{x\to 0}f^{-1}(x)=0$.

It is worth remarking that a related result appeared recently in \cite[Theorem 1.7]{MR4993440}, where it has been shown that there exists a subadditive bijection $f: H\to H$ satisfying \eqref{eq:mirkclaim} whenever $H$ is a positive rational cone generated by a $\mathbb{Q}$-linearly independent infinite subset of $(0,\infty)$. 
Although both arguments use $\mathbb Q$-linear independence to control possible additive
relations, our strategy of construction is substantially different. 
In fact, in \cite[Theorem 1.7]{MR4993440} the map is obtained by modifying countably many rational rays while leaving the complement fixed. 
On the other hand, in our case, the domain is the whole interval $(0,\infty)$, and the main object is an \textquotedblleft exotic\textquotedblright\, $\mathbb Q$-linear subspace $V$ of $\mathbb{R}^2$ (see Definition \ref{def:exotic} below) which is forced, by a transfinite construction, to be the graph of an additive function $T$ and to make an explicit map $\Phi$ bijective from a certain subset of $V$ to the whole $(0,\infty)$. Thus subadditivity here comes from the additivity of $T$ and the subadditivity of each section of $\Phi$, rather than from piecewise-linear concave maps on rays of $H$.

\section{Proof of Theorem \ref{thm:mainjanusz}}

Throughout, define the maps $\psi: \mathbb{R}\to \mathbb{R}$ and $\Phi:\mathbb R^2\to\mathbb R$ by 
\[
\psi(v):=\frac{|v|}{1+|v|}
\qquad \text{ and }\qquad 
\Phi(u,v):=u+\psi(v)
\]
for all $v \in \mathbb{R}$ and $(u,v) \in \mathbb{R}^2$. 
Let also $\pi_1:\mathbb R^2\to\mathbb R$ be the first-coordinate projection. 

For each $\mathbb Q$-linear subspace $V$ of $\mathbb R^2$, define
\[
V^+:=V\cap((0,\infty)\times\mathbb R).
\]

\begin{defi}\label{def:exotic}
A $\mathbb Q$-linear subspace $V$ of $\mathbb R^2$ is said to be \emph{exotic} if
\[
V\cap(\{0\}\times\mathbb R)=\{(0,0)\}
\quad \text{ and }\quad 
V\cap(\mathbb R\times\mathbb Q)=\{(0,0)\}.
\]
\end{defi}

In the proof of Theorem \ref{thm:mainjanusz}, we will need the following two intermediate results.

\begin{prop}\label{prop:technical1}
Let $W$ be a $\mathbb Q$-linear exotic subspace of $\mathbb R^2$ such that
$|W|<\mathfrak c$. Assume that $\Phi$ is one-to-one on $W^+$.

Fix $a\in\mathbb R\setminus\pi_1[W]$ and, for each $b\in\mathbb R$, define 
$$
W_b:=W+\mathbb Q(a,b).
$$ 
Then
\[
B_a:=
\left\{
b\in\mathbb R:
W_b\text{ is exotic and }\Phi\text{ is one-to-one on }W_b^+
\right\}
\]
is dense in $\mathbb R$.
\end{prop}

\begin{proof}
We need to show that every nonempty open interval contains a point of $B_a$. Fix a nonempty open
interval $J\subseteq\mathbb R$ and, for each $b\in J$, define $p_b:=(a,b)$.
We shall show that fewer than $\mathfrak c$ values of $b\in J$ do not belong to $B_a$. 

\medskip

\textbf{Exoticity.} 
Let us start checking the first exoticity condition for $W_b$ from Definition \ref{def:exotic}. 
Pick $w\in W$ and $q\in\mathbb Q$ and suppose that 
$
w+qp_b\in\{0\}\times\mathbb R. 
$ 
Taking first coordinates gives 
$
\pi_1(w)+qa=0.
$ 
If $q\ne0$, then
\[
a=-\frac{\pi_1(w)}q\in\pi_1[W],
\]
which contradicts the choice of $a$. Hence $q=0$, and then 
$
w\in W\cap(\{0\}\times\mathbb R)=\{(0,0)\}.
$ 
Thus no nonzero vertical vector is ever created.

\medskip

Next, we consider the second exoticity condition for $W_b$ from Definition \ref{def:exotic}. 
If it fails, it would mean
that, for some $w=(u,v)\in W$, $q\in\mathbb Q$, and $r\in\mathbb Q$, one has
\[
v+qb=r.
\]
If $q=0$, then $v=r\in\mathbb Q$, and hence 
$
w\in W\cap(\mathbb R\times\mathbb Q)=\{(0,0)\}.
$ 
So this gives no nonzero violation. If $q\ne0$, then 
$
b=(r-v)/q.
$ 
For each fixed triple $(w,q,r)\in W\times\mathbb Q\times\mathbb Q$, this excludes at most
one value of $b$. Since 
$
|W\times\mathbb Q\times\mathbb Q|<\mathfrak c,
$ 
fewer than $\mathfrak c$ values of $b\in J$ are excluded by this condition.

\medskip

\textbf{Injectivity.} It remains to preserve injectivity of $\Phi$ on $W_b^+$. For each $i \in \{1,2\}$, fix $w_i=(u_i,v_i)\in W$ and $q_i\in\mathbb Q$, 
and suppose that 
\begin{equation}\label{eq:collision}
\Phi(w_1+q_1p_b)=\Phi(w_2+q_2p_b).
\end{equation}
Equivalently,
\[
u_1+q_1a+\psi(v_1+q_1b)
=
u_2+q_2a+\psi(v_2+q_2b).
\]

\medskip

\textsc{Case 1.} 
First, if $q_1=q_2=0$, then \eqref{eq:collision} is simply $\Phi(w_1)=\Phi(w_2)$. 
If both points belong to $W^+$, the assumed injectivity of $\Phi$ on $W^+$ gives
$w_1=w_2$. Thus no new positive collision is created.

\medskip

\textsc{Case 2.} 
Now, suppose that exactly one of $q_1,q_2$ is zero, let us say $q_1=0$ and $q_2\ne0$. Add,
if necessary, the single point of $J$ where 
$
v_2+q_2b=0
$ 
to the forbidden set $J\setminus B_a$, and split the remaining part of $J$ into finitely many open
intervals on which $v_2+q_2b$ has constant sign. On each such interval, both sides of
the preceding equation are rational functions of $b$. The left-hand side is constant,
whereas the right-hand side contains the nonconstant term $\psi(v_2+q_2b)$. Thus the
equation cannot be an identity of rational functions, and hence it has only finitely many
solutions on each such interval. The case $q_2=0$ and $q_1\ne0$ is analogous.

\medskip

\textsc{Case 3.} 
It remains to consider the case $q_1q_2\ne0$. Split $J$ into finitely many open intervals
on which both affine functions
\[
b\mapsto v_1+q_1b
\quad \text{ and }\quad 
b\mapsto v_2+q_2b
\]
have constant sign. The finitely many endpoints, where one of these affine functions
vanishes, may also be added to the forbidden set  $J\setminus B_a$. 
On each remaining interval, both sides
of \eqref{eq:collision} are rational functions of $b$. Hence the equation has only
finitely many solutions on such an interval unless it is an identity of rational functions.

To complete the proof, we now discuss the identity case. On one of the intervals of fixed signs, for each $i \in \{1,2\}$ choose
$\varepsilon_i\in\{-1,1\}$ so that $\varepsilon_i(v_i+q_ib)>0$. 
Then
\[
\psi(v_i+q_ib)
=
1-\frac1{1+\varepsilon_i(v_i+q_ib)}.
\]
Thus the identity has the form
\[
A_1-\frac1{D_1(b)}
=
A_2-\frac1{D_2(b)},
\]
where
$
A_i:=u_i+q_ia+1$ and $
D_i(b):=1+\varepsilon_i(v_i+q_ib)$. 
If $D_1\ne D_2$, then
\[
A_1-A_2
=
\frac1{D_1(b)}-\frac1{D_2(b)}
=
\frac{D_2(b)-D_1(b)}{D_1(b)D_2(b)},
\]
which is a nonconstant rational function of $b$, a contradiction. Hence
$
D_1=D_2
$ and 
$A_1=A_2$. 
From $D_1=D_2$ we obtain
\[
\varepsilon_1q_1=\varepsilon_2q_2,
\qquad
\varepsilon_1v_1=\varepsilon_2v_2,
\]
while $A_1=A_2$ gives
\[
(u_1-u_2)+(q_1-q_2)a=0.
\]
Since $u_1-u_2\in\pi_1[W]$ and $a\notin\pi_1[W]$, it follows that $q_1=q_2$, 
and then $u_1=u_2$. Therefore $\varepsilon_1=\varepsilon_2$ and $v_1=v_2$. Hence
$w_1=w_2$ and $q_1=q_2$.

Consequently, for each fixed quadruple $(w_1,w_2,q_1,q_2)$, we need to exclude only finitely many values of $b\in J$. Since 
$ 
|W\times W\times\mathbb Q\times\mathbb Q|<\mathfrak c,
$ 
fewer than $\mathfrak c$ values of $b\in J$ are excluded in total from $B_a$. 

\medskip

Combining the two exoticity conditions and the injectivity condition \eqref{eq:collision}, we have excluded fewer than $\mathfrak c$ values of $b\in J$. Since $|J|=\mathfrak c$, we conclude that $B_a\cap J\neq \emptyset$, i.e., there exists
$b\in J$ such that $W_b$ is exotic and $\Phi$ is one-to-one on $W_b^+$. As $J$ was arbitrary,
$B_a$ is dense in $\mathbb R$.
\end{proof}

\begin{prop}\label{prop:technical2}
Let $W$ be a $\mathbb Q$-linear exotic subspace of $\mathbb R^2$ such that
$|W|<\mathfrak c$. Assume that $\Phi$ is one-to-one on $W^+$.

Fix $y\in(0,\infty)\setminus\Phi[W^+]$. Then there exists $p\in(0,\infty)^2$ such that
$\Phi(p)=y$ and
\[
W_p:=W+\mathbb Qp
\]
is an exotic $\mathbb Q$-linear subspace such that $\Phi$ is one-to-one on $W_p^+$.
\end{prop}

\begin{proof}
For each $t>0$, put
$$
p_t:=(y-\psi(t),t).
$$
Define also the nonempty open interval $I_y:=\{t\in (0,\infty): y>\psi(t)\}$. Observe that $p_t\in(0,\infty)^2$ for every $t\in I_y$, and $\Phi(p_t)=y$. As in the previous proof, 
we shall show that fewer than $\mathfrak c$ values of $t\in I_y$ are forbidden for the claimed property of $W_{p_t}$. 

\medskip

\textbf{Exoticity.} 
Let us check the first exoticity condition for $W_{p_t}$ from Definition \ref{def:exotic}. Suppose that 
$ 
w+qp_t\in\{0\}\times\mathbb R
$ 
for some $w=(u,v)\in W$ and $q\in\mathbb Q$. Taking first coordinates gives
\[
u+q\left(y-\psi(t)\right)=0.
\]
If $q=0$, then $u=0$, and so $w\in W\cap(\{0\}\times\mathbb R)=\{(0,0)\}$. 
If $q\ne0$, the preceding equation determines at most one value of $t\in I_y$ (since $\psi$ is strictly increasing on $(0,\infty)$). Hence fewer
than $\mathfrak c$ values of $t$ are excluded by the first condition.

Next, we consider the second exoticity condition for $W_{p_t}$ from Definition \ref{def:exotic}. A violation would mean that, for some
$w=(u,v)\in W$, $q\in\mathbb Q$, and $r\in\mathbb Q$, one has
\[
v+qt=r.
\]
With the same reasoning as in the proof of Proposition \ref{prop:technical1}, 
fewer than $\mathfrak c$ values of $t\in I_y$ are excluded by this second condition. 

\medskip

\textbf{Injectivity.} It remains to preserve injectivity of $\Phi$ on $W_{p_t}^+$. For each $i \in \{1,2\}$, fix
$w_i=(u_i,v_i)\in W$ and $q_i\in\mathbb Q$, and consider the equation
\begin{equation}\label{eq:collision2}
\Phi(w_1+q_1p_t)=\Phi(w_2+q_2p_t).
\end{equation}

\medskip

\textsc{Case 1.} If $q_1=q_2=0$, then \eqref{eq:collision2} is simply 
$
\Phi(w_1)=\Phi(w_2).
$ 
If both points belong to $W^+$, the assumed injectivity of $\Phi$ on $W^+$ gives
$w_1=w_2$. Thus no new positive collision is created.

\medskip

\textsc{Case 2.} 
Now, suppose that 
$q_1=0$ and $q_2\ne0$.
Add, if necessary, the single point of $I_y$ where 
$
v_2+q_2t=0
$ 
to the forbidden set, and split the remaining part of $I_y$ into finitely many open
intervals on which $v_2+q_2t$ has constant sign. On each such interval, both sides of
\eqref{eq:collision2} are rational functions of $t$. As in the proof of Proposition \ref{prop:technical1}, the equation has only finitely
many solutions on such an interval unless it is an identity of rational functions.

We now discuss the identity case for this latter case: since $q_1=0$ and $q_2\ne0$, from \eqref{eq:collision2} 
we have to understand when the function
\[
t\mapsto \Phi(w_2+q_2p_t)
\]
can be constant. On one of the intervals of fixed sign, choose
$\varepsilon\in\{-1,1\}$ so that 
$
\varepsilon(v_2+q_2t)>0.
$ 
Then
\begin{equation}\label{eq:representationPhi}
\begin{split}
\Phi(w_2+q_2p_t)
&=
u_2+q_2\left(y-1+\frac1{1+t}\right)
+
1-\frac1{1+\varepsilon(v_2+q_2t)}
\\
&=
u_2+q_2(y-1)+1+\frac{q_2}{1+t}
-\frac1{1+\varepsilon(v_2+q_2t)}.
\end{split}
\end{equation}
If this function is constant, the pole at $t=-1$ must be cancelled by the second
denominator. Hence
\[
1+\varepsilon(v_2-q_2)=0,
\]
or equivalently,
$
v_2=q_2-\varepsilon\in\mathbb Q$. 
Since $W$ is exotic, this gives $w_2=(0,0)$, and therefore $v_2=0$. Thus $q_2=\varepsilon$. 
If $q_2=1$ and $\varepsilon=1$, then $w_2+q_2p_t=p_t$, and therefore
$$\Phi(w_2+q_2p_t)=\Phi(p_t)=y$$ for every $t\in I_y$. 
If $q_2=-1$ and $\varepsilon=-1$, then
\[
\Phi(-p_t)=2-y-\frac2{1+t},
\]
which is not constant.

Therefore, in the case where $q_1=0$ and $q_2\ne0$, an identity can affect injectivity on
$W_{p_t}^+$ only if $w_2+q_2p_t$ is precisely $p_t$. In that case 
identity \eqref{eq:collision2} becomes $\Phi(w_1)=\Phi(p_t)=y$. 
Since this identity affects injectivity on $W_{p_t}^+$ only when both points belong to
$W_{p_t}^+$, then $w_1 \in W^+$. Hence we would have
\[
w_1\in W^+
\qquad\text{and}\qquad
\Phi(w_1)=y,
\]
which contradicts
$y\notin\Phi[W^+]$. Thus, no new bad cases are created in the case $q_1=0$ and $q_2\ne0$. 
The case $q_2=0$ and $q_1\ne0$ is analogous.

\medskip

\textsc{Case 3}. 
It remains to consider the case $q_1q_2\ne0$. 
As in the proof of Proposition \ref{prop:technical1}, split $I_y$ into finitely many open intervals
on which both $t\mapsto v_1+q_1t$ and $t\mapsto v_2+q_2t$ have constant sign. Add to the forbidden set those finitely many endpoints where these affine functions vanish. On each remaining interval, both sides
of \eqref{eq:collision2} are rational functions of $t$. Hence the equation has only
finitely many solutions on such an interval unless it is an identity of rational functions.

Hence, to complete the proof, we discuss the identity case. On one of the intervals of fixed signs, for each $i\in\{1,2\}$ choose
$\varepsilon_i\in\{-1,1\}$ so that $\varepsilon_i(v_i+q_it)>0$, and put 
$
D_i(t):=1+\varepsilon_i(v_i+q_it).
$ 
Then by \eqref{eq:representationPhi} we have 
\[
\begin{split}
\Phi(w_i+q_ip_t)
&
=
u_i+q_i(y-1)+1+\frac{q_i}{1+t}-\frac1{D_i(t)}.
\end{split}
\]

There is one exceptional possibility, namely, that the zero of $D_i$ coincides with
$t=-1$. This means 
$ 
D_i(-1)=0,
$ 
or equivalently,
\[
1+\varepsilon_i(v_i-q_i)=0.
\]
Thus 
$
v_i=q_i-\varepsilon_i\in\mathbb Q.
$ 
Since $W$ is exotic, we get $w_i=(0,0)$, hence $v_i=0$ and $q_i=\varepsilon_i$. Therefore
only two cases can occur. If 
$
q_i=\varepsilon_i=1
$ 
then $w_i+q_ip_t=p_t$ and $\Phi(w_i+q_ip_t)= y$. 
If 
$
q_i=\varepsilon_i=-1
$ 
then $w_i+q_ip_t=-p_t$, which has negative first coordinate for every $t\in I_y$.

Hence, for a genuine positive collision, the only constant exceptional case is the point
$p_t$. If one side is $p_t$, then the other side must also be identically equal to $y$.
By the preceding paragraph, this forces the other side to be the same point $p_t$ as well.
Thus no collision between distinct positive points is created by the exceptional cases.

We are left with the case where neither $D_1$ nor $D_2$ vanishes at $t=-1$. Then each
side has two distinct possible poles: one at $t=-1$, coming from the term
\[
\frac{q_i}{1+t},
\]
and one at the zero of $D_i(t)$. If the two rational functions are identical, then their
poles and residues must agree. Comparing the residues at $t=-1$ gives 
$
q_1=q_2.
$ 
Comparing the other pole and its residue gives $\varepsilon_1q_1=\varepsilon_2q_2$ and $\varepsilon_1v_1=\varepsilon_2v_2$. 
Since $q_1=q_2\ne0$, it follows that 
$\varepsilon_1=\varepsilon_2$ and $v_1=v_2$.
Finally, comparing the constant terms gives 
$
u_1=u_2.
$
Therefore
\[
w_1=w_2
\quad\text{and}\quad
q_1=q_2.
\]
So the two points $w_1+q_1p_t$ and $w_2+q_2p_t$ were identical.

Consequently, for each fixed quadruple $(w_1,w_2,q_1,q_2)$, we need to exclude only
finitely many values of $t\in I_y$. Since $|W\times W\times\mathbb Q\times\mathbb Q|<\mathfrak c$, 
fewer than $\mathfrak c$ values of $t\in I_y$ are excluded in total.

\medskip

Combining the two exoticity conditions and the injectivity condition \eqref{eq:collision2},
we have excluded fewer than $\mathfrak c$ values of $t\in I_y$. Since $|I_y|=\mathfrak c$,
there exists $t\in I_y$ such that, with $p:=p_t$, the space 
$
W_p:=W+\mathbb Qp
$ 
is exotic and $\Phi$ is one-to-one on $W_p^+$. By construction, $p\in(0,\infty)^2$ and
$\Phi(p)=y$. This completes the proof.
\end{proof}

We are finally ready to prove Theorem~\ref{thm:mainjanusz}. 
\begin{proof}[Proof of Theorem~\ref{thm:mainjanusz}]
Let us start noting that $\psi$ is subadditive. Indeed, for $a,b\ge0$,
\[
\psi(a)+\psi(b)-\psi(a+b)
=
\frac{ab(2+a+b)}{(1+a)(1+b)(1+a+b)}
\ge0.
\]
Since $\psi$ is increasing on $[0,\infty)$ and $|s+t|\le |s|+|t|$, it follows
that
\[
\forall s,t\in\mathbb R,\qquad
\psi(s+t)\le\psi(s)+\psi(t).
\]

\medskip

\textbf{A special countable exotic subspace.} We build a countable exotic space which will force the desired oscillation at the
origin. To this aim, let 
\[
W_\star^{(0)}:=\{(0,0)\}.
\]
and define the sequence 
$(W_\star^{(n)}: n \in \mathbb{N})$ 
by induction on $n$ as it follows. 

\smallskip

(i) 
Suppose that $W_\star^{(2n-2)}$ has already been constructed for some integer $n\ge 1$. Pick 
\[
a_n^+\in(0,\nicefrac{1}{n})\setminus\pi_1[W_\star^{(2n-2)}],
\]
which is possible since $|\pi_1[W_\star^{(2n-2)}]|<\mathfrak c$. By
Proposition~\ref{prop:technical1}, applied with $a=a_n^+$, the set of admissible parameters
is dense in $\mathbb R$. Hence we may choose 
$
t_n^+\in(n,\infty)
$ 
such that, putting
\[
W_\star^{(2n-1)}:=W_\star^{(2n-2)}+\mathbb Q(a_n^+,t_n^+),
\]
the space $W_\star^{(2n-1)}$ is exotic and $\Phi$ is one-to-one on
$(W_\star^{(2n-1)})^+$. 

\smallskip

(ii) 
Next, suppose that $W_\star^{(2n-1)}$ has already been constructed for some integer $n\ge 1$. Similarly, pick some 
\[
a_n^-\in(0,\nicefrac{1}{n})\setminus\pi_1[W_\star^{(2n-1)}].
\]
Again by Proposition~\ref{prop:technical1}, applied with $a=a_n^-$, we may choose 
$
t_n^-\in(-\nicefrac{1}{n},\nicefrac{1}{n})
$ 
such that, putting
\[
W_\star^{(2n)}:=W_\star^{(2n-1)}+\mathbb Q(a_n^-,t_n^-),
\]
the space $W_\star^{(2n)}$ is exotic and $\Phi$ is one-to-one on $(W_\star^{(2n)})^+$. 

\smallskip

Accordingly, define 
\[
W_\star:=\bigcup_{n\in\mathbb N}W_\star^{(n)}.
\]
Then $W_\star$ is a countable exotic $\mathbb Q$-linear subspace of $\mathbb R^2$ and $\Phi$ is
one-to-one on $W_\star^+$. In addition, for every positive integer $n$, we have $(a_n^+,t_n^+)\in W_\star$ and 
$(a_n^-,t_n^-)\in W_\star$, where 
\begin{equation}\label{eq:uglyconstraints}
a_n^+,a_n^- \in (0,\nicefrac{1}{n}), \quad 
t_n^+ \in (n,\infty), \quad \text{ and }\quad 
t_n^- \in (-\nicefrac{1}{n},\nicefrac{1}{n}).
\end{equation}
%
%


\medskip

\textbf{A transfinite chain of exotic subspaces}. 
Fix two enumerations
\[
(0,\infty)=\{y_\alpha:\alpha<\mathfrak c\}
\quad \text{ and }\quad 
\mathbb R=\{r_\alpha:\alpha<\mathfrak c\}.
\]
We construct, by transfinite recursion, an increasing chain 
$
(W_\alpha:\alpha\le\mathfrak c)
$ 
of exotic $\mathbb Q$-linear subspaces of $\mathbb R^2$ such that $\Phi$ is one-to-one on
$W_\alpha^+$ for every $\alpha\le\mathfrak c$.

\medskip

To this end, put
\[
W_0:=W_\star.
\]

\smallskip

At a limit ordinal $\lambda\le\mathfrak c$, set
\[
W_\lambda:=\bigcup_{\alpha<\lambda}W_\alpha.
\]
Since the chain is increasing, $W_\lambda$ is a $\mathbb Q$-linear subspace. Moreover,
$W_\lambda$ is exotic, because any vector of $W_\lambda$ lying in
$\{0\}\times\mathbb R$ or in $\mathbb R\times\mathbb Q$ already lies in some earlier
$W_\alpha$. Similarly, $\Phi$ is one-to-one on $W_\lambda^+$, because any collision between
two points of $W_\lambda^+$ already occurs in some $W_\alpha^+$.

\smallskip

Suppose now that $W_\alpha$ has been constructed. First, we force the value $y_\alpha$. If 
$
y_\alpha\in\Phi[W_\alpha^+],
$ 
put 
$$
W_\alpha':=W_\alpha. 
$$ 
Otherwise, by Proposition~\ref{prop:technical2}, it is possible to choose a point 
$
p_\alpha\in(0,\infty)^2
$ 
such that 
$
\Phi(p_\alpha)=y_\alpha
$ 
and such that
\[
W_\alpha':=W_\alpha+\mathbb Qp_\alpha
\]
is exotic and $\Phi$ is one-to-one on $(W_\alpha')^+$. 

Next, we force the first coordinate $r_\alpha$. If 
$
r_\alpha\in\pi_1[W_\alpha'],
$ 
put
\[
W_{\alpha+1}:=W_\alpha'.
\]
Otherwise, by Proposition~\ref{prop:technical1}, applied with $a=r_\alpha$, we may choose
$b_\alpha\in\mathbb R$ such that, putting 
$
q_\alpha:=(r_\alpha,b_\alpha),
$ 
the space
\[
W_{\alpha+1}:=W_\alpha'+\mathbb Qq_\alpha
\]
is exotic and $\Phi$ is one-to-one on $W_{\alpha+1}^+$. (It is worth observing that, at every stage $\alpha<\mathfrak c$, the space $W_\alpha$ is generated over $\mathbb Q$ by
fewer than $\mathfrak c$ many vectors; hence 
$
|W_\alpha|<\mathfrak c,
$ 
and Propositions~\ref{prop:technical1} and~\ref{prop:technical2} are indeed applicable.) 

\medskip

\textbf{A maximal exotic subspace as graph of an additive map.} 
Finally, we set
\[
V:=W_{\mathfrak c}.
\]
Then $V$ is an exotic $\mathbb Q$-linear subspace of $\mathbb R^2$, and $\Phi$ is one-to-one
on $V^+$. Moreover,
\[
\pi_1[V]=\mathbb R.
\]
Indeed, if $r\in\mathbb R$, then $r=r_\alpha$ for some $\alpha<\mathfrak c$, and at stage
$\alpha$ either $r_\alpha$ already belonged to $\pi_1[W_\alpha']$, or a vector with first
coordinate $r_\alpha$ was adjoined. Similarly,
\begin{equation}\label{eq:conditionVplus}
\Phi[V^+]=(0,\infty),
\end{equation}
because every $y_\alpha\in(0,\infty)$ was considered and, if it was not already attained,
a positive point with $\Phi$-value $y_\alpha$ was adjoined.

Since $V$ is exotic $\mathbb{Q}$-linear subspace of $\mathbb{R}^2$ and $\pi_1[V]=\mathbb R$, it follows that $V$ is the graph of a unique 
additive function
\[
T:\mathbb R\to\mathbb R.
\]
In fact, $\pi_1[V]=\mathbb{R}$ implies that $V\cap (\{x\}\times \mathbb{R})\neq \emptyset$ for all $x \in \mathbb{R}$. In addition, given $x,y,y^\prime \in \mathbb{R}$ such that $(x,y),(x,y^\prime) \in V$, we have $(x,y)-(x,y^\prime) \in V$, hence by the exoticity of $V$ we obtain $y=y^\prime$. It follows that it is possible to define $T: \mathbb{R}\to \mathbb{R}$ such that 
\begin{equation}\label{eq:VgraphT}
V=\mathrm{Graph}(T).
\end{equation}
Lastly, $T$ is additive. In fact, pick $(x, T(x)), (y,T(y)) \in V$ and note that $(x,T(x))+(y,T(y)) \in V$. Since the latter vector coincides with $(x+y,T(x+y))$, it follows that $T(x+y)=T(x)+T(y)$.

\medskip

\textbf{Final construction.} 
To conclude our proof, define the map $f: (0,\infty) \to (0,\infty)$ by 
$$
\forall x>0,\qquad
f(x):=\Phi(x,T(x)).
$$
We claim that $f$ witnesses the claimed properties.

\medskip

\textsc{Bijectivity of $f$.} 
Thanks to \eqref{eq:VgraphT} and \eqref{eq:conditionVplus}, we have
\[
f[(0,\infty)]=\Phi[\mathrm{Graph}(T)^+]=\Phi[V^+]=(0,\infty), 
\]
hence $f$ is surjective. In addition, since $\Phi$ is one-to-one on $V^+$, it follows also that $f$ is one-to-one. 

\medskip

\textsc{Subadditivity of $f$.} 
By the subadditivity of $\psi$ and the additivity of $T$, we get
\[
\begin{split}
f(x+y)
&=
x+y+\psi(T(x+y))
=
x+y+\psi(T(x)+T(y))
\\
&
\le
x+y+\psi(T(x))+\psi(T(y))
=
f(x)+f(y)
\end{split}
\]
for each $x,y>0$. 
Thus $f$ is subadditive.

\medskip

\textsc{Limit superior at $0$.} 
Taking into account $\psi<1$, we have 
$f(x)=x+\psi(T(x)) <x+1$ for all $x>0$. This implies that 
\[
\limsup_{x\to0}f(x)\le1.
\]
On the other hand, we get by construction that 
$$
(a_n^+, t_n^+) \in W^{(2n-1)}_\star\subseteq W_\star=W_0\subseteq W_{\mathfrak{c}}=V=\mathrm{Graph}(T)
$$
for all $n\ge 1$. Since $\psi$ is increasing on $[0,\infty)$, and recalling that $T(a_n^+)=t_n^+>n$ and $a_n^+>0$ by \eqref{eq:uglyconstraints}, we obtain 
$$
\forall n\ge 1, \quad 
f(a_n^+)=
a_n^++\frac{|T(a_n^+)|}{1+|T(a_n^+)|}>\frac{|T(a_n^+)|}{1+|T(a_n^+)|}>\frac{n}{n+1}.
$$
Again by \eqref{eq:uglyconstraints} we have $\lim_n a_n^+=0$, hence
$$
\limsup_{x\to 0}f(x) 
\ge \limsup_{n\to \infty}f(a_n^+)
\ge \limsup_{n\to \infty}\frac{n}{n+1}=1.
$$
Therefore $\limsup_{x\to 0}f(x) =1$. 

\medskip

\textsc{Limit inferior at $0$.} It proceeds similarly as in the previous case. Since by construction we have $(a_n^-,t_n^-) \in W_\star^{(2n)} \subseteq \mathrm{Graph}(T)$ for all $n\ge 1$, it follows by \eqref{eq:uglyconstraints} that 
$$
\forall n\ge 1, \quad 
f(a_n^-)=a_n^-+\frac{|T(a_n^-)|}{1+|T(a_n^-)|}<a_n^-+|T(a_n^-)|=a_n^-+|t_n^-|<\frac{2}{n}.
$$
Taking into account that $\lim_n a_n^-=0$, we conclude that $\liminf_{x\to 0}f(x)=0$. This completes the proof of Theorem \ref{thm:mainjanusz}.  
\end{proof}

\section{Further properties of the construction}\label{sec:final} 

As it follows from the construction in the proof of Theorem \ref{thm:mainjanusz}, if $x>0$, then $(x,T(x)) \in V$ and so $T(x) \notin \mathbb{Q}$. Taking into account that $0\le \psi<1$, we conclude that the map $f$ satisfies also 
$$
\forall x>0, \qquad x<f(x)=x+\psi(T(x))<x+1.
$$

\medskip

Finally, it is worth noting that, given a real $k>0$, there exists a subadditive bijection $g: (0,\infty)\to (0,\infty)$ such that $x<g(x)<x+k$ for all $x>0$ and, in addition, 
$$
\liminf_{x\to 0}g(x)=0\qquad \text{ and }\qquad 
\limsup_{x\to 0}g(x)=k.
$$
Indeed, it is enough to consider the map $g(x):=kf(x/k)$. 

\section{Concluding remarks}\label{sec:conclusions}

In this concluding section, we provide some details about the connection between Theorem~\ref{thm:mainjanusz} and a conjectural strengthening of a converse Minkowski-type theorem. 

To this aim, given a measure space $(\Omega,\Sigma,\mu)$, let $\mathcal{S}_{\mu}$ be the set of all nonnegative $\mu$-integrable simple functions $h:\Omega\to\mathbb R$. 
Denote the support of each $h \in \mathcal{S}_{\mu}$ by $\mathrm{supp}(h):=\{\omega\in \Omega: h(\omega)\neq 0\}$. Moreover, given a bijection $\varphi: (0,\infty)\to (0,\infty)$, define the functional $\bm{P}_\varphi: \mathcal{S}_\mu\to [0,\infty)$ by 
$$
\forall h \in \mathcal{S}_\mu, \qquad 
\bm{P}_\varphi(h):=
\begin{cases}
\displaystyle
\,\varphi^{-1}\left(\int_{\mathrm{supp}(h)}\varphi\circ |h|\,\mathrm{d}\mu\right)\,\,
& \text{if } \mu(\mathrm{supp}(h))>0,\\[2ex]
\,0
& \text{if } \mu(\mathrm{supp}(h))=0.
\end{cases}
$$
Then $\bm{P}_\varphi$ is well defined, and similar in shape to the $L^p$ norm. Accordingly, Matkowski proved in \cite[Theorem 1]{MR1009994} the following characterization: 
\begin{thm}\label{thm:matk3874}
    Let $(\Omega,\Sigma,\mu)$ be a measure space for which 
    \begin{equation}\label{eq:measurabledisjoint}
    \exists A,B \in \Sigma,\quad 
    0<\mu(A)<1<\mu(B)<\infty
    \end{equation}
    Pick also a bijection $\varphi: (0,\infty)\to (0,\infty)$ such that $\lim_{x\to 0}\varphi^{-1}(x)=0$. 

    If the functional $\bm{P}_\varphi$ is subadditive then there exists $p\ge 1$ such that $\varphi(t)=\varphi(1)t^p$ for all $t \in (0,\infty)$. 
\end{thm}
In other words, the subadditivity of  $\bm{P}_\varphi$ implies that $\bm{P}_\varphi$ is the $L^p$ norm for some $p\ge 1$, providing a converse to Minkowski's inequality. 
In the proof, the subadditivity of $\bm{P}_\varphi$ and \eqref{eq:measurabledisjoint} implies that $f:=\varphi^{-1}$ is a subadditive bijection on $(0,\infty)$. Taking into account the hypothesis $\lim_{x\to 0}\varphi^{-1}(x)=0$ (that is, $\lim_{x\to 0}f(x)=0$) and using the main result in \cite{MR1088646}, it follows that $f$ is an increasing homeomorphism of $(0,\infty)$. 
This fact, playing a key role in the proof of the above theorem, led to the natural question of whether condition $\lim_{x\to 0}\varphi^{-1}(x)=0$ could be weakened. 
Question \ref{q:matk2} concerns the regularity step used in Matkowski’s proof: informally, once subadditivity of $\bm P_\varphi$ yields subadditivity of $f=\varphi^{-1}$, the assumption $\lim_{x\to0}f(x)=0$ forces a good behavior of $f$. Theorem \ref{thm:mainjanusz} shows that the weaker condition $\liminf_{x\to0}f(x)=0<\limsup_{x\to0}f(x)<\infty$ is insufficient by itself to force continuity.

In connection with Theorem \ref{thm:matk3874}, recall that, if $f: (0,\infty)\to (0,\infty)$ is a subadditive injection and there exists a set $C\subseteq (0,\infty)$ such that 
$$
\lim_{x\to 0}(f\upharpoonright C)(x)=0
\quad \text{ and }\quad 
\liminf_{x\to 0^+}\frac{\lambda(C\cap [0,x])}{x}>0,
$$
then $f$ is continuous, cf. \cite[Theorem 6]{MR2928997}; here $\lambda$ stands for the Lebesgue measure.

Lastly, it is worth remarking that several alternatives which replace the condition $\lim_{x\to 0}\varphi^{-1}(x)=0$ in Theorem \ref{thm:matk3874} have been studied in \cite{MR1400540}. 

\section{Acknowledgments} The author is grateful to Janusz Matkowski for proposing Question \ref{q:matk2} and providing all the detailed motivations included in Section \ref{sec:conclusions}. 

\subsection{Declaration of competing interest} The author confirms that no data have been used for the preparation of the manuscript. In addition, he declares that there is no conflict of interest and that no
fundings were received.

\subsection{Data availability} No data was used for the research described in the article.

\bibliographystyle{amsplain}

\end{document}